\documentclass{amsart}
\usepackage[T1]{fontenc}
\usepackage[utf8]{inputenc}
\usepackage{latexsym}
\usepackage{amsfonts}
\usepackage{amsmath}
\usepackage{amssymb}
\usepackage{amsthm}
\usepackage[nobysame,initials]{amsrefs}
\usepackage{xcolor}
\definecolor{darkred}{RGB}{180,0,0}
\definecolor{darkblue}{RGB}{0,0,180}
\usepackage{hyperref}
\hypersetup{
  colorlinks = true,
  urlcolor   = darkblue,
  linkcolor  = darkblue,
  citecolor  = darkred,
  pdftitle = {Straightening warped cones},
  pdfauthor = {Damian Sawicki and Jianchao Wu},
  pdfkeywords = {warped cone, coarse embedding, Haagerup property, proper actions on Banach spaces, amenable action, property A, asymptotically faithful map}
}
\usepackage[capitalise]{cleveref}
\usepackage{enumitem}

\usepackage{ulem}
\usepackage{upref}

\newtheorem{thm}{Theorem}[section]
\newtheorem{theorem}[thm]{Theorem}
\newtheorem*{thm*}{Theorem}
\newtheorem{lem}[thm]{Lemma}
\newtheorem{XxmpX}[thm]{Lemma}
\newenvironment{lem-no-proof}    
  {%
   \pushQED{\qed}\begin{XxmpX}}
  {\popQED\end{XxmpX}}

\newtheorem{prop}[thm]{Proposition}
\newtheorem{cor}[thm]{Corollary}
\newtheorem{ex}[thm]{Example}
\newtheorem{question}[thm]{Question}
\newtheorem{thmx}{Theorem}

\newtheorem{thmi}{Theorem}

\theoremstyle{definition}
\newtheorem{defi}[thm]{Definition}
\newtheorem{longRem}[thm]{Remark}

\def\R{{\mathbb R}}
\def\C{{\mathbb C}}
\def\N{{\mathbb N}}
\def\Z{{\mathbb Z}}

\def\T{{\mathbb T}}
\def\SS{{\mathbb S}}
\def\LL{{\mathbb L}}

\def\U{{\mathcal U}}
\def\cO{{\mathcal O}}

\def\eps{\varepsilon}
\def\st{\ |\ }

\DeclareMathOperator{\supp}{supp}
\DeclareMathOperator{\diam}{diam}
\DeclareMathOperator{\Prob}{Prob}
\DeclareMathOperator{\id}{id}

\newcommand{\acts}{\ensuremath\mathrel{\raisebox{.6pt}{$\curvearrowright$}}}
\newcommand*{\defeq}{\mathrel{\vcenter{\baselineskip0.5ex \lineskiplimit0pt
                     \hbox{\scriptsize.}\hbox{\scriptsize.}}}%
                     =}
\newcommand*{\eqdef}{=\mathrel{\vcenter{\baselineskip0.5ex \lineskiplimit0pt
                     \hbox{\scriptsize.}\hbox{\scriptsize.}}}}
\newcommand*\dif{\mathop{}\!\mathrm{d}}

\let\oldeprint\eprint
\newcommand{\neweprint}[1]{\href{https://arxiv.org/abs/#1}{\tt arXiv:#1}}

\begin{document}
\title{Straightening warped cones}

\author{Damian Sawicki}
\address{Institute of Mathematics of the Polish Academy of Sciences
    \newline \'Sniadeckich 8, 00-656 Warszawa, Poland}
\curraddr{Max-Planck-Institut f\"{u}r Mathematik
    \newline Vivatsgasse 7, 53111 Bonn, Germany}
\urladdr{\href{https://guests.mpim-bonn.mpg.de/dsawicki/}{guests.mpim-bonn.mpg.de/dsawicki/}}

\author{Jianchao Wu}
\address{Department of Mathematics, Penn State University, University
Park PA 16802}
\curraddr{Mailstop 3368,
Department of Mathematics,
Texas A\&M University,
College Station, Texas 77843,
United States}
\urladdr{\href{https://www.math.tamu.edu/~jwu/}{www.math.tamu.edu/~jwu/}}


\begin{abstract}
We provide the converses to two results of J. Roe (Geom.\ Topol.\ 2005): first, the warped cone associated to a free action of an a-T-menable group admits a fibred coarse embedding into a Hilbert space, and second, a free action yielding a warped cone with property A must be amenable. We construct examples showing that in both cases the freeness assumption is necessary. The first equivalence is obtained also for other classes of Banach spaces, in particular for $L^p$-spaces.
\end{abstract}

\maketitle

\section{Introduction}

Given an action $\Gamma\acts Y$ of a finitely generated group on a compact metric space, J. Roe \cite{Roe-cones} constructs an unbounded metric space $\cO_\Gamma Y$ {\textemdash} the warped cone over the action {\textemdash} encoding dynamical properties of the action in its coarse structure. The outline of the construction is as follows {\textemdash} first, take the space and scale its metric by a large constant, then, fixing a finite set of generators for the group, we declare that whenever a point is mapped to another by the action of a generator, the distance between these two points shall be no more than $1$. 
The metric obtained after creating these ``shortcuts'' is called the warped metric and the resulting family of metric spaces (indexed by the positive scaling constants) is called the warped cone. 

Dynamical and ergodic properties of the action are reflected in large-scale geometric properties of the warped cone. In particular, the action has a spectral gap if and only if levels of the warped cone are quasi-isometric to an expander graph \cite{Vigolo}, even if one considers Banach-space expanders (including so-called super-expanders) and spectral gaps \cite{Sawicki-superexp} (see also \cite{Nowak--Sawicki} for the first results in this direction).

Recall the main results of \cite{Roe-cones}:

\begin{thmi}[Roe]
\label{Roe1}
Assume that the action $\Gamma \acts Y$ is amenable. Then $\cO_\Gamma Y$ has property A.
\end{thmi}
\begin{thmi}[Roe]
\label{Roe2}
Assume that $\Gamma < G$ is a subgroup in a compact Lie group $G$. If~$\cO_\Gamma G$ has property A, then $\Gamma$ is amenable.
\end{thmi}
\begin{thmi}[Roe]
\label{Roe3}
Assume that $\Gamma$ is a subgroup in a compact Lie group $G$. If~$\cO_\Gamma G$ embeds coarsely into a Hilbert space, then $\Gamma$ has the Haagerup property, i.e.\ it is a-T-menable.
\end{thmi}

\Cref{Roe3} remains true for fibred coarse embeddings and also more general classes of actions \cite{Sawicki-completions}.

The main goal of this paper is to provide the converses to \cref{Roe1} and~\labelcref{Roe3} (we also generalise \cref{Roe2} and~\labelcref{Roe3}). The following result (\cref{A-to-amenable} in text) gives the converse of \cref{Roe1}.

\begin{thmx}\label{theorema} Assume that the action $\Gamma \acts Y$ is free. If $\cO_\Gamma Y$ has property A, then the action must be amenable.
\end{thmx}
Since the kind of actions in \cref{Roe2} always carry at least one invariant probability measure (that is, the Haar measure), a condition which makes amenability of the action equivalent to amenability of the acting group, thus 
\cref{theorema} also generalises \cref{Roe2}. 

\Cref{Roe3} has the following converse (\cref{Haagerup-to-fibred} in text).

\begin{thmx}\label{theoremb}
Assume that the action $\Gamma \acts Y$ is free and {linearisable} (e.g.\ it is the action of a subgroup on an ambient compact Lie group) and that $\Gamma$ is a-T-menable. Then $\cO_\Gamma Y$ fibred coarsely embeds into a Hilbert space.
\end{thmx}

In fact, \cref{theoremb} holds for more general classes of Banach spaces (where a-T-menability is replaced by the property of admitting a proper action on an appropriate Banach space), and we also provide a new proof of \cref{Roe3} valid in the Banach-space setting (\cref{fibred-to-PL^p}). Please also note that the word ``fibred'' cannot be removed in \cref{theoremb}, as there are many subgroup actions of a-T-menable groups yielding warped cones that do not admit a coarse embedding into a Hilbert space \cite{Nowak--Sawicki}.

We also provide the following generalised version of \cref{Roe3} (\cref{coarse-to-a-T-menable} in text).

\begin{thmx}\label{theoremc}
Assume the action $\Gamma \acts Y$  is free and $\cO_\Gamma Y$ admits a coarse embedding into a Hilbert space. Then the action is a-T-menable.
\end{thmx}

In Theorems \labelcref{theorema}, \labelcref{theoremb}, and~\labelcref{theoremc} we assume that the action is free, and in \cref{sec:freeness} we construct examples showing that the assumption is necessary in all these results. \Cref{sec:problems} discusses some open problems.

\subsection*{Some context} For a residually finite group, one can consider a descending chain of finite index normal subgroups. The respective inverse system of finite quotients can be thought of as a finitary approximation of the initial infinite subgroup. This approach has led to a lot of interesting results and research directions, like the celebrated L\"{u}ck approximation theorem \cite{Lueck}, the study of profinite rigidity, rank gradient, residual finiteness growth, etc.

In the geometric context, such a sequence of finite quotients constitutes the so-called box space. Box spaces provided the first explicit examples of expander graphs \cite{Margulis}, the first examples of bounded geometry metric spaces without property A admitting a coarse embedding into a Hilbert space \cite{AGS}, and the first examples of spaces coarsely non-embeddable into a Hilbert space yet not containing expanders \cite{AT}, also among large-girth graphs \cite{DK}. If one constructs an infinite group containing such box spaces \cite{Osa}, it has exotic properties, like being a counterexample to strong versions of the Baum--Connes conjecture \cites{Gromov, AD, HLS} or being a-T-menable but not exact \cites{AO, Osa}.

Warped cones can be seen as providing finite (or compact, to be precise) approximations of not necessarily residually finite groups. Another well-studied notion of approximations for such groups are sofic approximations, studied recently in the geometric context in \cite{AFS}. It can be thought that box spaces and sofic approximations represent the combinatorial (discrete) part and warped cones the dynamical or ergodic part of the same theory.

The first-named author showed that for a box space one can construct an action (namely, the action on a profinite completion) such that the resulting warped cone contains the box space quasi-isometrically and retains its properties such as coarse embeddability into a Hilbert space \cite{Sawicki-completions}. Hence, warped cones can be viewed as a generalisation of box spaces.

Roe pointed out that there is a parallel between how a box space and a warped cone is related to the group it comes from and, indeed, via the construction of \cite{Sawicki-completions} one can use results on warped cones to recover the known results on box spaces. The case of amenability follows from versions of Theorems \labelcref{Roe1} and~\labelcref{Roe2} obtained in \cite{Sawicki-completions}, but the case of a-T-menability is more involved and requires our \cref{theoremb} (together with a version of \cref{Roe3}).

It is worth noting that while for box spaces property A is equivalent to amenability of the group, for warped cones it is equivalent to amenability of the action, which is a much weaker notion in general.

\subsection*{Acknowledgements}
The authors are very grateful to Piotr Nowak for connecting them to each other when they were thinking about matters of the present paper. We thank Alexander Engel for many helpful remarks on the manuscript. We are grateful to an anonymous referee whose comments made us restructure some proofs to a great benefit of the reader. The second-named author is grateful to Qin Wang for introducing him to some of the problems discussed in this paper. 

The first-named author is grateful to Max Planck Institute for Mathematics in Bonn, where he was a post-doc when the final version of the manuscript was prepared, for its hospitality and financial support.

The first-named author was partially supported by Narodowe Centrum Nauki grant Preludium number 2015/19/N/ST1/03606. The second-named author was partially supported by SFB 878 \emph{Groups, Geometry and Actions} and NSF grant \#DMS--1564401.

\section{Definitions}\label{sec:defn}

Let $Y$ be a compact subset of a sphere $\SS^n\subseteq \R^{n+1}$ with a continuous action of a group $\Gamma$ coming with a finite set of generators $S$. We form the \emph{infinite cone} over~$Y$, namely $\cO Y=\{ry \st r>0, y\in Y \}\subseteq \R^{n+1}$, and equip it with the Euclidean metric $d$ and with the obvious extension of the action on $Y$. The \emph{warped cone}, denoted $\cO_\Gamma Y$, is the infinite cone equipped with the warped metric $d_\Gamma$, which is the largest metric bounded by the initial metric $d$ and satisfying $d_\Gamma(x,sx)\leq 1$, for any $x\in \cO Y$ and $s\in S$.

When the metric space $(Y,d)$ does not come with an embedding into a sphere in a normed space, we can define the infinite cone to be $\R_+ \times Y$ equipped with the metric $d$ given by $d((r,y),\, (r',y')) = |r-r'| \cdot \operatorname{diam}(Y) + \min(r,r') \cdot d(y,y')$. The above norm metric and this $\ell^1$-metric are Lipschitz equivalent (in particular coarsely equivalent) if $(Y,d)$ happens to be a subset of a sphere, so without loss of generality we will stick to this $\ell^1$-metric in more technical statements.

For brevity, we will often denote the pair $(r,y)$ by $ry$, even if $Y$ does not come with an embedding into a normed space.

We will also need some definitions from large scale geometry and the theory of group actions. 

The following notion was introduced by Gromov.

\begin{defi} A function $f\colon X\to Y$ between metric spaces is called a \emph{coarse embedding} if there exist non-decreasing unbounded functions $\rho_-,\rho_+\colon \R_+\to \R_+$ such that:
\[\rho_-\circ d_X(x,x')\leq d_Y(f(x),f(x')) \leq \rho_+\circ d_X(x,x')\,.\]
\end{defi}

The existence of a coarse embedding of a metric space $X$ into a well-behaved metric space $Y$, like a Hilbert space, is a useful regularity condition with strong consequences, as illustrated by the following theorem due to G. Yu \cite{Yu}.

\begin{theorem}
Let $X$ be a discrete metric space with bounded geometry. If $X$ admits a coarse embedding into a Hilbert space, then the coarse Baum--Connes conjecture holds for $X$.
\end{theorem}

As a consequence, for $X$ being a finitely generated group, coarse embedding into a Hilbert space implies the strong Novikov conjecture \cites{Yu, STY}.

As a generalisation of coarse embeddings, Chen, Wang, and Yu \cite{CWY} introduced fibred coarse embeddings and proved the \emph{maximal} coarse Baum--Connes conjecture for spaces admitting fibred coarse embeddings into a Hilbert space.

\begin{defi}\label{def:fibred}
A metric space $X$ is said to admit a \textit{fibred coarse embedding} into a normed vector space $E$ if there exists
\begin{itemize}
\item a field of isometric copies of $E$ over $X$: $\{ E_{x} \}_{x \in X}$;
\item a sequence of sections $f_n\colon X \rightarrow \bigsqcup_{x \in X}E_{x}$ (i.e.\ $f_n(x) \in E_{x}$);
\item two non-decreasing unbounded functions $\rho_{\pm}\colon[0,\infty)\to[0,\infty)$
\end{itemize}
such that for any $n\in\N$ there exists a bounded subset $K_{n}\subseteq X$ and a trivialisation:
\[\tau_n^x\colon \bigsqcup_{y\in B(x,n)} E_{y} \to B(x,n)\times E\]
for every $x \in X \setminus K_{n}$ (that is, $\tau^x_n(E_y) = \{y\} \times E$). We denote the composition of $\tau^x_n$ with the projection onto $E$ by $t^x_n$ and require that the restriction of $t^x_n$ to any $E_y$ for $y\in B(x,n)$ is an isometry onto $E$ such that:
\begin{itemize}
\item for any $y,y' \in B(x,n)$: 
\[\rho_-\circ d(y,y')\leq \| t^x_n\circ f_n(y) - t^x_n\circ f_n(y') \Vert \leq \rho_+\circ d(y,y')\,;\]
\item for any two points $x,x'\in X\setminus K_n$ such that the intersection $B(x,n)\cap B(x',n) \eqdef I_n^{x,x'}$ is non-empty, there exists an isometry $t_n^{x,x'}\!\colon E\to E$ such that $\mathrm{id}\times t_n^{x,x'}\! \colon I_n^{x,x'}\times E \to I_n^{x,x'}\times E$ is equal to the composition $\tau_n^x \circ (\tau_n^{x'})^{-1}$.
\end{itemize} 
\end{defi}

Note that in \cite{CWY} $f_n$ is not allowed to depend on $n$, but, by altering $\tau^x_n$, one can require all $f_n$ to be simply the constant zero section, hence the two definitions are equivalent.

The equivariant (group-theoretic) counterpart of the above properties is the Haagerup property, also known as a-T-menability of Gromov.

\begin{defi} Let $\mathcal E$ be a family of Banach spaces. A finitely generated group $\Gamma$ has \emph{property $P\mathcal E$} if it admits a proper isometric action by affine maps on a Banach space $E\in \mathcal E$. For $\mathcal E$ being the class of Hilbert spaces, we say that $\Gamma$ has \emph{the Haagerup property} or that $\Gamma$ is \emph{a-T-menable}.
\end{defi}
Recall that an isometric action on a Banach space is called \emph{(metrically) proper}, if for every orbit map, inverse images of balls are finite. In other words, every (equivalently: one) orbit map is a coarse embedding.

In \cref{Haagerup-to-fibred} we will need the following property.

\begin{defi} An action $\Gamma \acts Y$ is \emph{linearisable} in a Banach space $E$ if there exists an isometric representation of $\Gamma$ on $E$ and a bi-Lipschitz equivariant embedding $Y\to E$.
\end{defi}

In \cref{sec:linearisable} we verify this condition for different types of actions including subgroup actions (in particular actions on profinite completions) and isometric actions on manifolds.

\section{From a-T-menable and \texorpdfstring{$PL^p$}{PL\^{}p} groups to embeddable warped cones}

Recall a theorem of Roe \cite{Roe-cones} as generalised in \cite{Sawicki-completions}.

\begin{thm}\label{Roe-embeddable}
	Let $\Gamma\acts Y$ be an action of a finitely generated group on a compact metric space. 
	Assume the action admits an invariant probability measure and is essentially free. If the warped cone $\cO_\Gamma Y$ admits a (fibred/as\-ymptotic\footnote{By admitting an \emph{asymptotic coarse embedding} we mean admitting an \emph{asymptotically conditionally negative definite kernel} in the terminology of \cite{BaumGuentnerWillett}. This condition is a weakening of coarse embeddability. See also \cite{randomGraphs}, from which our terminology is borrowed.}) coarse embedding into a Hilbert space, then $\Gamma$ has the Haagerup property, i.e.\ it is a-T-menable.
\end{thm}

In this section, we provide a partial converse to the above result: 

\begin{theorem}\label{Haagerup-to-fibred}
	Let $\Gamma\acts Y$ be an action of a finitely generated group on a compact metric space. 
Assume the action is free and linearisable in a Hilbert space, and that $\Gamma$ is a-T-menable. Then $\cO_\Gamma Y$ fibred coarsely embeds into a Hilbert space.
\end{theorem}

In fact, as it will be clear from the proof, the theorem has its counterpart for $L^p$-spaces (see \cref{PL^p-to-fibred}). Hence, one gets the following interesting application.

\begin{cor} Let $\Gamma$ be a free subgroup in a compact simple Lie group $G$ that acts on $G$ with a spectral gap as established in \cite{BdS}. Then, the warped cone $\cO_\Gamma G$ admits a fibred coarse embedding into an $L^p$-space for all $1\leq p<\infty$ even though it does not coarsely embed into any $L^p$-space.
\end{cor}
\begin{proof}
The embeddability part follows from \cref{Haagerup-to-fibred} and \cref{PL^p-to-fibred} as subgroup actions on compact Lie groups
are linearisable
by \cref{manifold}. Non-embeddability follows from the spectral gap by \cite{Nowak--Sawicki}.
\end{proof}

\begin{ex} Freeness of the action is crucial in \cref{Haagerup-to-fibred}. The action $\mathrm{SL}_2(\Z)\acts \T^2$ is essentially free and $\mathrm{SL}_2(\Z)$ is a-T-menable, but $\cO_{\mathrm{SL}_2(\Z)} \T^2$ does not fibred coarsely embed into any $L^p$-space for $1\leq p<\infty$.
\end{ex}
\begin{proof}
By the celebrated result of Selberg, $\mathrm{SL}_2(\Z)$ has property $(\tau)$ with respect to the congruence subgroups $\ker\left(\mathrm{SL}_2(\Z) \to \mathrm{SL}_2(\Z/n\Z)\right)$. As the $\mathrm{SL}_2(\Z)$-action on the orbit of $\left(\frac{1}{n},0\right)\in\T^2$ factors through $\mathrm{SL}_2(\Z/n\Z)$, the sequence of Schreier graphs $\mathrm{Sch}\left(\left(\frac{1}{n},0\right), \mathrm{SL}_2(\Z)\right)$ forms an expander. Since for each $n$ there exists $r_n$ such that $\mathrm{Sch}\left(\left(\frac{1}{n},0\right), \mathrm{SL}_2(\Z)\right)\times [r_n,\infty)$ embeds isometrically into $\cO_{\mathrm{SL}_2(\Z)} \T^2$ \cite{Sawicki-completions}*{Remark~3.1}, this is an obstruction for the existence of a fibred coarse embedding.
\end{proof}

One may argue that this action is not only non-free but also non-isometric and non-linearisable, so the example is not convincing evidence that freeness is important in \cref{Haagerup-to-fibred}. Such evidence will be given in \cref{ball-action}, which exhibits that the existence of \emph{one} fixed point within a linearisable and otherwise free action is enough for the conclusion of \cref{Haagerup-to-fibred} to fail.

\begin{longRem} The above example shows that the amenable case and the a-T-menable case differ. By \cite{Roe-cones} (see \cref{Roe-amenable-actions} below), under mild assumptions on~$Y$ and the Lipshitzness of the action, the warped cone $\cO_\Gamma Y$ always has property~A for an amenable $\Gamma$. However, for fibred coarse embeddability of the warped cone we need freeness of the action in addition to a-T-menability of the group.

This may be somehow reminiscent of the fact that amenability is inherited by quotients and a-T-menability is not. However, compare \cref{ball-action}, where there are no finite quotients involved.
\end{longRem}

We will need some preparation in order to obtain \cref{Haagerup-to-fibred}. The key technique is ``straightening'' the warped cone by the following construction. 

\begin{defi}\label{defi:straighten}
	Let $(X,d)$ be a metric space with a continuous action of a finitely generated group $\Gamma$, with $S$ being a finite set of generators. The \emph{twisted metric} $d^1$ on the Cartesian product $\Gamma \times X$ is the largest metric such that $d^1((\gamma, x), (s\gamma, x)) = 1$ and $d^1((\gamma, x), (\gamma, x')) \leq d(\gamma x, \gamma x')$ for all $s\in S\setminus \{e\}$, $\gamma\in \Gamma$, and $x\in X$.
\end{defi}

If $d$ is $\Gamma$-invariant, the metric $d^1$ is just the $\ell^1$-product metric of $d$ and the right-invariant word metric on $\Gamma$.

Note that the action on $(\Gamma \times X,d^1)$ given by $\gamma (\eta, x) = (\eta\gamma^{-1},\gamma x)$ is isometric and that the quotient space of this action can be identified with $X$ via the quotient map $(\gamma,x)\mapsto \gamma x$. Now recall from \cite{Roe-cones} that the warped metric $d_\Gamma$ for a general $\Gamma$-space $(X,d)$ (not necessarily a cone $\cO Y$) is the largest metric bounded by $d$ such that $d_\Gamma(sx, x)\leq 1$ for any $x\in X$  and $s\in S$. 

\begin{prop} The warped metric $d_\Gamma$ on $X$ is equal to the quotient metric of $d^1$, that is, for any $x, x' \in X$, 
	\begin{equation}\label{quotient metric}
		d_\Gamma(x,x') = \inf_\gamma d^1((e,x), (\gamma, \gamma^{-1}x'))\,.
	\end{equation}
\end{prop}
\begin{proof}
	Fix $x, x' \in X$. 
	Let us consider the left-hand side of \cref{quotient metric}. By \cite{Roe-cones}*{Proposition~1.6} it is equal to the infimum of \emph{mileages} of sequences of the form 
\[(x , \; x_1 , \; s_1 x_1 , \; x_2 , \; s_2 x_2 , \; \cdots , \; x_n , \; s_n x_n , \; x')\,,\]
where $s_i\in S$ and the mileage is defined as
\begin{equation}\label{distance downstairs}
d(x,x_1) + 1 + d(s_1x_1, x_2) + 1 + \cdots + 1 + d(s_nx_n, x') \,.
\end{equation}
The mileage formula corresponds to the two conditions $d_\Gamma\leq d$ and $d_\Gamma(x,sx)\leq 1$ in the definition of the warped metric.

The right-hand side of \cref{quotient metric} is the infimum of the distance $d^1((e, x), (\gamma, \gamma^{-1}x'))$ over $\gamma\in \Gamma$.
Similarly as above, this distance is the infimum of the mileages of sequences of the form
\begin{multline*}
(e,x) , \,(e, x_1) , \,(s_1, x_1) , \,(s_1, s_1^{-1}x_2) , \,(s_2 s_1, s_1^{-1}x_2) , \cdots \\
 \ldots , \,(s_{n-1}\ldots s_1, s_1^{-1}\ldots s_{n-1}^{-1} x_n) , \,(s_n \ldots s_1, s_1^{-1}\ldots s_{n-1}^{-1} x_n),\\
  (s_n \ldots s_1, s_1^{-1}\ldots s_n^{-1} x') = (\gamma, \gamma^{-1}x')\,,
\end{multline*}
where the mileage of such a sequence equals
\begin{multline}\label{distance upstairs}
d(x,x_1) + 1 + d(s_1\cdot x_1,\, s_1\cdot s_1^{-1} x_2) + 1 + \cdots \\
+ 1 + d((s_n \ldots s_1)\cdot s_1^{-1}\ldots s_{n-1}^{-1} x_n,\; (s_n \ldots s_1)\cdot s_1^{-1}\ldots s_n^{-1} x') \, ,
\end{multline}
where, again, the definition comes from the two conditions defining $d^1$.

It happens that the expressions in lines \eqref{distance downstairs} and \eqref{distance upstairs} are the same, which ends the proof.
\end{proof}

We remark that the minimum is always reached in \cref{quotient metric}, since for fixed $x$ and $x'$, we have $d^1((e,x), (\gamma, \gamma^{-1}x')) \to \infty$ as $\gamma \to \infty$.  
As an immediate consequence, we have: 

\begin{cor}\label{cor:ball-quotient}
	For any $x \in X$ and $m > 0$, the ball $B_{d_\Gamma}(x, m)$ under the warped metric is the image of the ball $B_{d^1}((e,x), m)$ in $\Gamma \times X$ under the quotient map $\Gamma \times X \to X$, $(\gamma, x) \mapsto \gamma x$. \qed
\end{cor}

We will now establish, under the assumption of freeness, the following crucial property for the quotient map $(\Gamma\times \cO Y,d^1) \to \cO_\Gamma Y$.

\begin{defi}[cf.\ \cite{expgirth1}]\label{def:asymptotically-faithful} A surjective map $q\colon Z\to X$ between metric spaces is said to be \emph{asymptotically faithful} if for every $N\in \N$ there is a subset $A\subseteq X$ with a bounded complement such that for any $z \in q^{-1}(A)$, the map $q$ restricts to an isometry between the balls $B(z, N)$ and $B(q(z), N)$.  
\end{defi}

\begin{prop}\label{faithful-lemma} The quotient map $(\Gamma\times \cO Y,d^1) \to \cO_\Gamma Y$  is asymptotically faithful if and only if the action $\Gamma\acts Y$ is free.
\end{prop}

The proof of \cref{faithful-lemma} will be split between a number of lemmata, which will also be utilised in \cref{sec:A}.
Before starting the proof, let us point out that the usefulness of the freeness condition is rooted in the following simple observation based on the compactness of $Y$.

\begin{lem-no-proof}\label{lem:tower}
	Let $\Gamma \acts Y$ be a free action on a compact space. Then for any $N \in \N$, there exists $\eps>0$ such that for any $y \in Y$, the collection of balls $\left\{  B_d(\gamma y,\eps) \st \gamma \in B(e, N) \right\}$ is disjoint; in other words, the map $q \colon \Gamma\times Y \to Y$, $(\gamma, y) \mapsto \gamma y$ is injective on the ``twisted rectangle'' 
	\[ \bigsqcup_{\gamma \in B(e, N)} \{\gamma\} \times \gamma^{-1} B_d(\gamma y,\eps)\,. \qedhere \]
\end{lem-no-proof}
 
When the action is isometric, the twisted rectangle above is simply the product $B(e, r) \times B_d(y,\eps)$ of balls. Even when the action is not isometric, thanks to the compactness of $Y$ one can still prove an analogous statement with this product of balls replacing the twisted rectangle, and the same can be said about \cref{lem:ball-in-box} below.

Next we explain how such a twisted rectangle is related to balls in $\Gamma \times \cO Y$ under the twisted metric $d^1$ coming from the metric on the cone $\cO Y$ defined in the second paragraph of \cref{sec:defn}. 

\begin{lem}\label{lem:ball-in-box}
	For any $N \in \mathbb{N}$ and $\eps > 0$, there is $R > 0$ such that for any $y \in Y$ and $r \geq R$, the ball $B_{d^1}((e,r , y), N)$ in $\Gamma \times \cO Y$ under the twisted metric is contained in the ``twisted cuboid'' 
	\[
		\bigsqcup_{\gamma \in B(e, N)} \{\gamma\} \times \left[r-\tfrac{N}{\operatorname{diam}(Y)},\; r+ \tfrac{N}{\operatorname{diam}(Y)} \right] \times  \gamma^{-1} B_d(\gamma y, \eps) \, , 
	\]
	where $d$ is the original unwarped metric on $Y$. 
\end{lem}

\begin{proof}
	Fix $N \in \mathbb{N}$ and $\eps > 0$. 
	By the compactness of $Y$ and the finiteness of the ball $B(e, N) \subseteq \Gamma$, we may find $\delta > 0$ such that for any $y_0, \ldots, y_N \in Y$ and any $\gamma_1, \ldots, \gamma_N \in B(e, N)$, we have 
	\[
		\sum_{i=1}^N d(\gamma_i y_{i-1},  \gamma_i y_{i}) \leq  \eps \qquad \text{whenever} \qquad \sum_{i=1}^N d( y_{i-1},  y_{i}) \leq \delta \, .
	\]
	Choose $R > N / \operatorname{diam}(Y)$ such that
	\begin{equation}\label{ineq:delta}
		\frac{N}{R - N / \operatorname{diam}(Y)}  \leq  \delta \, .	
	\end{equation}

	We claim this $R$ meets our requirement. To see this, we fix $y \in Y$ and $r \geq R$. Writing $x = ry$ for a point in the cone $\cO Y$, the ball $B_{d^1}((e,x), N)$ contains all points $(\gamma, \gamma^{-1}x')$ such that there is a sequence 
	\begin{multline*}
	(e,x) = (e,x_0),\, (e, x_1),\, (s_1, x_1),\, (s_1, s_1^{-1}x_2),\, (s_2 s_1, s_1^{-1}x_2),\, \cdots \\
	\cdots,\, (s_{n-1}\ldots s_1, s_1^{-1}\ldots s_{n-1}^{-1} x_n),\, (s_n \ldots s_1, s_1^{-1}\ldots s_{n-1}^{-1} x_n), \\
	(s_n \ldots s_1, s_1^{-1}\ldots s_n^{-1} x_{n+1}) = (\gamma, \gamma^{-1}x')
	\end{multline*}
	of mileage at most $N$, where $s_i\in S$ and $x_i \in \cO Y$.
	Recall that this mileage is $d(x,x_1) + 1 + d(s_1 x_1, x_2) + 1 + \cdots + 1 + d(s_n x_n, x)$, which equals:
	\[n + \sum_{i=0}^n \left( |r_i - r_{i+1}| \cdot \operatorname{diam}(Y) + \min(r_i, r_{i+1}) \cdot  d(s_i y_i, y_{i+1}) \right),\]
	where we write $x_i = r_i \, y_i \in \cO Y$ for $i\in \{0,\ldots, n+1\}$ and put $s_0=e$.  
	
	Hence, we have: 
	\begin{itemize}[itemsep=10pt]
		\item $n \leq N$, which implies $\gamma \in B(e, N)$, 
		\item $\sum_{i=0}^n  |r_i - r_{i+1}| \cdot \operatorname{diam}(Y) \leq N$, which implies by the triangle inequality $ |r - r_{i}| \leq \frac{N}{\operatorname{diam}(Y)}$ and thus also $ r_i \geq R - \frac{N}{\operatorname{diam}(Y)}$ for any $i$, and 
		\item $ \sum_{i=0}^n  \min(r_i, r_{i+1}) \cdot  d(s_i y_i, y_{i+1}) \leq N$, which implies 
		\[
			\sum_{i=0}^n  d(s_i y_i, y_{i+1}) \leq \frac{ \sum_{i=0}^n  \min(r_i, r_{i+1}) \cdot  d(s_i y_i, y_{i+1})}{ R - N / \operatorname{diam}(Y)}  \leq \delta
		\]
		(by \cref{ineq:delta}), and thus, writing $\gamma_i = s_n s_{n-1} \cdots s_i$ for $i = 1, \ldots, n$ and $\gamma_{n+1} = e$, by our choice of $\delta$ we have
		\[
			d(\gamma y_0, y_{n+1}) \leq \sum_{i=0}^n d(\gamma_{i} y_i, \gamma_{i+1} y_{i+1}) = \sum_{i=0}^n d(\gamma_{i+1}  s_i y_i, \gamma_{i+1}  y_{i+1}) \leq \varepsilon\,.
		\]
	\end{itemize}
	Taken together, the above bullet points show that 
	\begin{align*}
		(\gamma, \gamma^{-1}x') & = (\gamma, r_{n+1}, \gamma^{-1}y_{n+1}) \\
		& \in \bigsqcup_{\eta \in B(e, N)} \{\eta\} \times \left[r-\tfrac{N}{\operatorname{diam}(Y)}, r+ \tfrac{N}{\operatorname{diam}(Y)} \right] \times   \eta^{-1} B_d(\eta y, \eps) \,, 
	\end{align*}
	as desired. 
\end{proof}

The following is an immediate consequence of \cref{lem:ball-in-box} and \cref{cor:ball-quotient}. 

\begin{cor}\label{cor:ball-in-tower}
	For any $N \in \mathbb{N}$ and $\eps > 0$, there is $R > 0$ such that for any $y \in Y$ and $r \geq R$, the ball $B_{d_\Gamma}(r y, N)$ in the warped cone $\cO_\Gamma Y$ is contained in 
	\[
		\left[r-\tfrac{N}{\operatorname{diam}(Y)},\; r+ \tfrac{N}{\operatorname{diam}(Y)} \right] \times  \bigcup_{\gamma \in B(e, N)} B_d(\gamma y, \eps) \, ,
	\]
	where $d$ is the original unwarped metric on $Y$. 
	\qed
\end{cor}

We are now ready to prove \cref{faithful-lemma}. 

\begin{proof}[Proof of \cref{faithful-lemma}]
	It is straightforward to show that the freeness of the action $\Gamma\acts Y$ is implied by the asymptotic faithfulness of the quotient map $q \colon (\Gamma\times \cO Y,d^1) \to \cO_\Gamma Y$: if there is $\gamma\neq e$ and $y \in Y$ such that $\gamma y = y$, then points $(e,r,y)$ and $(\gamma, r, y)$ are mapped to the same point $(r,y)$, even though their distance is $|\gamma|$, which contradicts asymptotic faithfulness.

To prove the converse, we first claim that for any $N\in \N$ there exists $R_N > 0$ such that for any $r \geq R_N$ and $y \in Y$, the quotient map $q$ is injective on the ball $B_{d^1}((e,r , y),\, N)$. Indeed, since the action is free, by \cref{lem:tower}, there exists $\eps>0$ such that for any $y \in Y$, the map $\Gamma\times Y \to Y$, $(\gamma, y) \mapsto \gamma y$ is injective on the twisted rectangle $\bigsqcup_{\gamma \in B(e, N)} \{\gamma\} \times \gamma^{-1} B_d(\gamma y,\eps)$. Then by \cref{lem:ball-in-box}, there is $R > 0$ such that for any $y \in Y$ and $r \geq R$, the ball $B_{d^1}((e,r , y), N)$ in $\Gamma \times \cO Y$ under the twisted metric is contained in the twisted cuboid
\[
	\bigsqcup_{\gamma \in B(e, N)} \{\gamma\} \times \left[r-\tfrac{N}{\operatorname{diam}(Y)},\; r+ \tfrac{N}{\operatorname{diam}(Y)} \right] \times  \gamma^{-1} B_d(\gamma y, \eps) \, . 
\]
Combining these two facts proves the claim.

Now, to prove the quotient map $q$ is asymptotically faithful, it suffices to show for every $N$ there exists $R_N'$ such that the quotient map $q$ becomes isometric when restricted to $B_{d^1}((\gamma, x),\,N)$ with $x=ry$ when $r\geq R_n'$.
Since the diagonal action $\Gamma \acts \Gamma\times \cO Y$ is isometric, it suffices to consider balls centred at points of the form $(e,x)$ with $x \in \cO Y$. 
Let $(\gamma_1,x_1)$ and $(\gamma_2,x_2)$ be two points in the $N$-ball at $(e,x)$. They are mapped to points $\gamma_1x_1$ and $\gamma_2 x_2$ in the quotient, whose distance is $M = \min_\gamma d^1((\gamma_1,x_1), (\gamma_2\gamma,\gamma^{-1}x_2))$. We need to check that this equals $O=d^1((\gamma_1,x_1), (\gamma_2,x_2))$.
Clearly, $M\leq O\leq 2N$. Hence, by the triangle inequality we get that for the optimal $\gamma$ yielding the minimum we have $d^1((e,x),(\gamma_2\gamma, \gamma^{-1}x_2)\leq 3N$. Thus, if we take $R_N' \geq R_{3N}$ so that the quotient map $q$ is injective on $3N$-balls, we get that $\gamma$ must be the identity, which ends the proof.
\end{proof}

To prove \cref{Haagerup-to-fibred}, we will need one more result. 
\begin{theorem}[cf.\ \cite{CWY}*{Theorem 2.2}]\label{CWY}
Let $Z$ be a metric space with a free isometric action of $\Gamma$ such that the quotient map $Z\to Z/\Gamma$ is asymptotically faithful and that $Z$ admits an equivariant coarse embedding into a normed space $E$ with some isometric action. Then, $Z/\Gamma$ admits a fibred coarse embedding into $E$.
\end{theorem}

\begin{proof}[Proof of \cref{Haagerup-to-fibred}]
By \cref{faithful-lemma} the quotient map $\Gamma\times \cO Y \to \cO_\Gamma Y$ is asymptotically faithful. 
The equivariant bi-Lipschitz embedding $Y\to H_1$ can be extended to an equivariant bi-Lipschitz embedding $\cO Y \to \R \oplus H_1$ (where the action on $\R$ is trivial).
Similarly, as $\Gamma$ is assumed to be a-T-menable, there exists a Hilbert space $H_2$ with an isometric action of $\Gamma$ and an equivariant coarse embedding $\Gamma\to H_2$ given by an arbitrary orbit map $\gamma \mapsto \gamma \cdot \xi \in H_2$. Hence, the product $\Gamma \times \cO Y$ admits an equivariant coarse embedding into $\R \oplus H_1 \oplus H_2$. The claim follows from \cref{CWY}.
\end{proof}

By the same argument one can obtain similar results for any Banach space.

\begin{cor}\label{PL^p-to-fibred} Assume that the action $\Gamma \acts Y$ is free and linearisable in a
Banach space $E_1$, and that $\Gamma$ has property $PE_2$ for some Banach space $E_2$. Then $\cO_\Gamma Y$ fibred coarsely embeds into $\R \oplus E_1 \oplus E_2$.
\end{cor}

By combining the theorem of Roe (including its $L^p$-version, \cref{fibred-to-PL^p}, which we prove in \cref{sec:embeddable}) with the above results, we get the following.

\begin{cor}\label{PX-fibred-equivalence}
	Let $\Gamma\acts Y$ be an action of a finitely generated group on a compact metric space. 
Assume the action is free and linearisable in a Hilbert space (respectively, an $L^p$-space) and admits an invariant measure. Then the warped cone $\cO_\Gamma Y$ admits a fibred coarse embedding into a Hilbert space (respectively, an $L^p$-space) if and only if $\Gamma$ is a-T-menable (respectively, has property $PL^p$).
\end{cor}
\begin{proof}
The ``only if'' part follows from \cref{fibred-to-PL^p} by the assumption of the existence of an invariant measure (and of the essential freeness).

The ``if'' part follows from \cref{Haagerup-to-fibred} and \cref{PL^p-to-fibred}, and here the assumption of linearisability (and freeness) is used.
\end{proof}

\begin{longRem} The above \cref{PX-fibred-equivalence} holds also for more general classes of Banach spaces. The tools used in the proofs are Bochner spaces $L^p(Y,\mu; X)$ with coefficients in some Banach space $X$ (where $1\leq p\leq \infty$ can be chosen freely) and their ultrapowers (for \cref{fibred-to-PL^p}), and finite direct sums (for \cref{Haagerup-to-fibred}), so one can consider any classes of Banach spaces closed under these operations.
\end{longRem}

\subsection{Linearisable actions}\label{sec:linearisable}
In \cref{Haagerup-to-fibred} and \cref{PL^p-to-fibred} we assume that our $\Gamma$-space $Y$ admits an equivariant bi-Lipschitz embedding into a Hilbert or $L^p$-space. In this section we are going to verify this condition for certain actions, including the most prominent cases of profinite completions and manifolds. 

We will start from the most general result, which guarantees the existence of \emph{another} metric that is equivalent to the original metric but makes the action linearisable. Next, we will show linearisability with the original metric.

\begin{lem}\label{Kuratowski} Let $p\geq 1$ and let $\Gamma\acts Y$ admit a finite invariant measure $\mu$ of full support and an invariant metric $d$. Then, there exists an equivalent metric $d'_p$ such that the action on $(Y,d'_p)$ is linearisable in an $L^p$-space.
\end{lem}
\begin{proof}
Let $f\colon Y\to L^p(Y,\mu)$ be the Kuratowski embedding given by $(f(x))(y) = d(x,y)$ for any $x,y \in Y$. Checking the continuity and injectivity is a standard exercise. Because $Y$ is compact, $f$ is a homeomorphism onto its image. By the triangle inequality it also follows that $f$ is a Lipschitz embedding, but its inverse need not be Lipschitz in general. Hence, we define $d'_p$ to be the metric induced from $L^p(Y,\mu)$.

Finally, $f$ is equivariant: for any $x,y \in Y$, we have \[(f(\gamma x))(y) = d(\gamma x, y) = d(x, \gamma^{-1}y) = (f(x))(\gamma^{-1}y) = (\gamma f(x))(y)\,.\qedhere\]
\end{proof}

For example, the above construction does not give a Lipschitz-equivalent metric for a \emph{profinite completion} (more generally, for ultrametric spaces). Recall that given a group $\Gamma$ and a decreasing sequence of its finite index normal subgroups $\Gamma_n$, we can consider the inverse system of finite quotients $G_n=\Gamma/\Gamma_n$ and its limit $\varprojlim G_n$.

The inverse limit (known as the boundary of the coset tree or the profinite completion with respect to $(\Gamma_n)$) can be seen as the set $\{(g_n)\in \prod_n G_n \st q(g_{n+1})=g_n, \, \forall n\}$, where $q$ is the obvious quotient map $G_{n+1}\to G_n$. Consequently, it inherits the product metric $d((g_n),(h_n)) = a_j$, where $j$ is the smallest index such that $g_j\neq h_j$ and $(a_j)_j$ is a sequence of positive numbers decreasing to $0$. We will assume that the sequence $a_j$ decays at least geometrically.

\begin{lem} The action $\Gamma\acts \varprojlim G_n$ is linearisable in an $L^p$-space for any $1\leq p < \infty$.
\end{lem}
\begin{proof}
Let our $L^p$-space be defined as $L= \ell^p( \bigsqcup_n G_n)$. The action of $\Gamma$ permutes the coordinates by left translation on each quotient $G_n$. We define the embedding by $f((g_n)) = \sum_n 2^{-1/p}\cdot a_n\cdot\delta_{g_n}$. For two sequences $(g_n)$, $(h_n)$ as above we have:
\[\|f((g_n)) - f((h_n))\|_p = \left(\sum_{i=j}^\infty a_i^p\right)^{1/p} \asymp a_j = d((g_n),(h_n))\,,\]
where the approximate equality follows from the assumption of geometric decay.
\end{proof}

Let us now show that the general construction of \cref{Kuratowski} gives a bi-Lipschitz embedding in the most important case of manifolds.

\begin{lem}\label{manifold} If $Y$ is an $n$-dimensional Riemannian manifold, then the metric $d_p'$ from \cref{Kuratowski} is Lipschiz equivalent to the metric $d$ with Lipschitz constants not depending on $p$.
\end{lem}
\begin{proof}
As the distance function is $1$-Lipschitz, we also get Lipschitzness of the Kuratowski embedding: $d_p' \leq \mu(Y)^{1/p}\cdot d$, where $\mu$ is the Riemannian volume on $Y$.

In the other direction, we will separately consider two cases $d(x,y)\leq r/2$ and $d(x,y)>r/2$, where $r>0$ is such that for every point $x$ on the manifold $Y$, the exponential map $\phi\colon T_x Y \to Y$ is a diffeomorphism between $B(0,r)$ in $T_x Y$ and $B(x,r)$ in $Y$. We can also assume that $\phi$ is an isometry when restricted to any line segment containing $0$ and also approximately preserves the measure: $C^{-1}\mu(\phi(A))\leq \lambda(A)\leq C\mu(\phi(A))$, for some $C>1$, where $A$ is any measurable subset of $B(0,r)$ and $\lambda$ is the Lebesgue measure on $T_x Y$.

Using carefully the Taylor expansion, one can prove a version of the law of cosines on a Riemannian manifold \cite{DDD-CR}*{Lemma 3.2}: namely, that for $v,w\in B(0,r)$ we have
\[\Big| d(\phi(v),\phi(w))^2 - \|v\|^2 - \|w\|^2 + 2\langle v,w \rangle \Big| \leq K \cdot \|v\|^2\cdot \|w\|^2,\]
where $K<\infty$ is some constant depending only on $Y$ and $r$. Let now $x,y\in Y$ be such that $d(x,y)\leq r/2$ and let $v\in T_x Y$ satisfy $\phi(v)=y$. Consider now any $w\in B(0,r)\subseteq T_x Y$ such that $\cos\left(\measuredangle(v,w)\right)\leq -1/2$ and assume for simplicity $\|w\|\geq r/2$. We have:
\begin{align*}
d(y,\phi(w))^2 - d(x,\phi(w))^2 &= d(\phi(v),\phi(w))^2 - \|w\|^2 \\
&\geq \|v\|^2 - 2\langle v,w \rangle - K\|v\|^2\|w\|^2 \\
& \geq \|v\| \|w\| - K \|v\|^2 \|w\|^2 \\
&\geq \|v\| (r/2 - K r^4/4) \geq r/4 \cdot \|v\| = r/4 \cdot d(x,y)
\end{align*}
(if we shrink $r$ so that $r^3<K^{-1}$), and consequently
\[d(y,\phi(w)) - d(x,\phi(w)) \geq \frac{r/4 \cdot d(x,y)}{d(y,\phi(w)) + d(x,\phi(w))} \geq \frac{r/4 \cdot d(x,y)}{1.5r+r} = \frac{d(x,y)}{10}\,.\]
Clearly, the set $W\subseteq B(0,r)$ of $w$ as above has a positive measure $\lambda(W)$ (note that this value does not depend on $x$ or $y$), and hence $\mu(\phi(W))\geq C^{-1}\lambda(W)$. Thus we get:
\begin{align*}
d_p'(x,y)^{p} &= \int_Y |d(x,z) - d(y,z)|^p \dif\mu(z) \\
&\geq \int_{\phi(W)} |d(x,z) - d(y,z)|^p \dif\mu(z) \\
&\geq C^{-1}\lambda(W) \cdot \frac{d(x,y)^p}{10^p}\,.
\end{align*}

The remaining estimate is straightforward. If $d(x,y)\eqdef d>r/2$, we obtain:
\[d_p'(x,y) \geq \int_{B(x,d/3)} |d(x,z) - d(y,z)|^p \dif\mu(z) \geq \mu(B(x,d/3)) \cdot \left(\frac{d}{3}\right)^p,\]
and we know that $\mu(B(x,d/3)) \geq \mu(B(x,r/6)) \geq C^{-1} \lambda(B(0,r/6))$.
\end{proof}

The assumptions on linearisability of actions in \cref{Haagerup-to-fibred} and \cref{PL^p-to-fibred} may seem restrictive with our very general definition of a warped cone. However, in practise one is interested in nice spaces and actions, in particular the original formulation of \cref{Roe-embeddable} considered only actions of dense subgroups on ambient compact Lie groups \cite{Roe-cones}. The above \cref{manifold} shows, in particular, that such actions are linearisable.

In fact, for any compact group we have the following argument. (Note that instead of assuming a bi-Lipschitz embedding of $Y$, one can assume a coarse embedding of $\cO Y$ and obtain an equivariant coarse embedding of $\cO Y$, which suffices for \cref{Haagerup-to-fibred} and \cref{PL^p-to-fibred} as a replacement of linearisability.)

\begin{lem} Let $G$ be a compact group with a left-invariant metric, a (right-invariant) Haar measure $\mu$, and a bi-Lipschitz embedding into an $L^p$-space $L$ for some $p\geq 1$. Then its left translation action on itself is linearisable in an $L^p$-space.
\end{lem}
\begin{proof}
Let $c\colon G\to L$ be the embedding. Consider $L^p(G,\mu; L)$ with the Koopman representation $\pi$ of $G$ induced by the action by right translations: $(\pi_g v)(h) = v(hg)$, where $g,h\in G$ and $v\in L^p(G,\mu; L)$. Then, the embedding $\iota\colon G\to L^p(G,\mu; L)$ given by $(\iota(g))(h)=c(hg)$ for $g,h\in G$ is equivariant and bi-Lipschitz.
\end{proof}

\section{From warped cones with property A to amenable actions}\label{sec:A}

Recall the following result of Roe \cite{Roe-cones}.
\begin{theorem}\label{Roe-amenable-actions}
Let $\Gamma\acts Y$ be an amenable action by Lipschitz homeomorphisms of a finitely generated group on a compact manifold or a finite simplicial complex. Then $\cO_\Gamma Y$ has property A.
\end{theorem}
In fact, it follows from the proof that instead of assuming that $Y$ is a manifold or a simplicial complex, it is enough to assume that the infinite cone $\cO Y$ has property A (which holds for all examples considered in the literature so far, in particular for profinite completions).

In this section, we want to obtain the converse implication.

\begin{theorem}\label{A-to-amenable}
Let $\Gamma\acts Y$ be a free action of a finitely generated group on a compact metric space and assume that $\cO_\Gamma Y$ has property A. Then the action $\Gamma\acts Y$ is amenable.
\end{theorem}

Note that one needs to assume freeness of the action. In \cref{non-amenable-action} and~\labelcref{non-a-T-menable-action} we construct warped cones with property A over non-amenable and even non-a-T-menable actions.

Let us recall the definitions.

\begin{defi}\label{def:property-A} A metric space $X$ has \emph{property A}, if there is a sequence of maps $A_n\colon X \to \Prob(X)$ and a function $N\colon \N\to \N$ such that $\supp A_n(x) \subseteq B(x,N(n))$ and for $d(x,x')\leq n$ we have $\|A_n(x)-A_n(x')\|_1 \leq 1/n$.
\end{defi}

Property A was introduced by Yu \cite{Yu} as a property that implies coarse embeddability into a Hilbert space but is stronger and easier to work with, and the above characterisation comes from \cite{HigsonRoe}. While coarse embeddability is a non-equivariant version of a-T-menability (the Haagerup property) of groups, property A should be thought of as a non-equivariant version of amenability. Note that maps $A_n$ do not have to be continuous, but the continuity can always be imposed by a partition-of-unity argument.

\begin{defi}\label{def:amenable-action} An action $\Gamma \acts Y$ is amenable if there is a sequence of continuous maps $Y\owns y\mapsto C_n^y \in \Prob(\Gamma)$, which are asymptotically equivariant:
\[\lim_{n\to \infty} \sup_{y\in Y} \|C_n^{y}\gamma^{-1} - C_n^{\gamma y}\|_1 = 0\,,\]
for all $\gamma\in \Gamma$, where the right action on $\Prob(\Gamma)$ is defined by $(\mu \gamma) (A) = \mu(A \gamma^{-1})$ for $\mu\in\Prob(\Gamma)$ and $A\subseteq \Gamma$.
\end{defi}

Note that a group has property A if and only if it admits an amenable action on a compact Hausdorff space,
and it is amenable if and only if all of its actions are amenable.

\begin{proof}[Proof of \cref{A-to-amenable}]
Assume that the warped cone $\cO_\Gamma Y$ has property A, i.e.\ there is a sequence of \emph{continuous} (see the discussion below \cref{def:property-A}) maps $A_n\colon\cO_\Gamma Y \to \Prob(\cO_\Gamma Y)$ and a function $N\colon \N\to \N$ as in \cref{def:property-A}. Without changing the notation, we modify $A_n$ to satisfy the following conditions instead: 
\begin{enumerate}
	\item\label{A-to-amenable-proof-support} for any $ry \in \cO_\Gamma Y$, the support of $A_n(ry)$ is contained in $r Y \cap B_{d_\Gamma}(ry,N(n))$, 
	\item\label{A-to-amenable-proof-flat} for any $r\in \R_+$ and $y,y' \in Y$ with $d_\Gamma(ry, ry') \leq n$, we have $\|A_n(ry)-A_n(ry')\|_1 \leq 1/n$, and
	\item\label{A-to-amenable-proof-continuous} the map $Y \ni y \mapsto A_n(ry)$ is continuous. 
\end{enumerate}
In other words, the new $A_n$ splits into a family $\left\{ A_n |_{rY} \colon rY \to \Prob(rY) \right\}_{r \in \R_+}$ of continuous maps simultaneously satisfying the conditions in \cref{def:property-A}.
This can be arranged by considering the retractions $\pi_r\colon \cO_\Gamma Y \to rY$ given by $\pi_r(r'y) = ry$ for $y \in Y$ and $r, r' \in \R_+$, replacing each $A_n(ry)$ by its push-forward under $\pi_r$. Notice that since we are using the warped metric associated to the cone metric defined in the second paragraph of \cref{sec:defn}, we have ${d_\Gamma}(\pi_r(r'y'), \pi_r(ry)) = {d_\Gamma}(ry', ry) \leq {d_\Gamma}(r'y', ry)$ for any $r,r'\in (0,\infty)$ and $y,y'\in Y$, that is, $\pi_r\left( B_{d_\Gamma}(ry,N(n)) \right) \subset r Y \cap B_{d_\Gamma}(ry,N(n))$ for any $r\in (0,\infty)$ and $y\in Y$, which establishes the first property above.

We are going to produce a sequence of continuous maps $(y \mapsto C_n^y)$ as in \cref{def:amenable-action}. 
To this end, we fix $n \in \N$ and let $m = N(n)$. 
Since the action $\Gamma \acts Y$ is free, by \cref{lem:tower}, there exists $\eps>0$ such that 
\begin{enumerate}\setcounter{enumi}{3}
	\item\label{A-to-amenable-proof-disjoint} for any $y \in Y$, the balls $B_d(\gamma y,\eps)$, for $\gamma \in B(e, m+2)$, are disjoint, 
\end{enumerate}
where $d$ is the original unwarped metric on $Y$. Next, by \cref{cor:ball-in-tower}, there is $r > 0$ such that for any $y \in Y$, the ball $B_{d_\Gamma}(r y, m)$ in the warped cone $\cO_\Gamma Y$ is contained in 
\[
	\left[r-\tfrac{m}{\operatorname{diam}(Y)},\; r+ \tfrac{m}{\operatorname{diam}(Y)} \right] \times  \bigcup_{\gamma \in B(e, m)} B_d(\gamma y, \eps/2) \; .
\]
Combining this with the first property of $A_n$ detailed at the beginning of the proof and writing $r B_d(\gamma y,\eps/2)$ as a short form for $\{ r \} \times B_d(\gamma y,\eps/2)$, we have
\begin{enumerate}\setcounter{enumi}{4}
	\item\label{A-to-amenable-proof-balls} the support of $A_n(ry)$ is contained in $\bigcup_{\gamma \in B(e, m)} r B_d(\gamma y,\eps/2)$.
\end{enumerate}

We define, for any $y \in Y$ and $\gamma \in \Gamma$, 
\[
	C_n^y(\{\gamma\}) = 
	\begin{cases}
		A_n(ry) \left( r B_d(\gamma y,\eps) \right) & \textrm{if }\gamma\in B(e,m+2)\,, \\
		0 & \textrm{if }\gamma \in \Gamma \setminus B(e,m+2)	\, .
	\end{cases}
\]

We now check that it satisfies the desired properties. In the following, the numbers above the equality and inclusion signs indicate which of the aforementioned properties we are using.  
\begin{enumerate}[itemsep=5pt]
	\setcounter{enumi}{5}
	\item For any $y \in Y$, we have 
	\begin{align*}
		C_n^y(\Gamma) = \sum_{\gamma \in B(e, m+2) } C_n^y(\{\gamma\}) & = \sum_{\gamma \in B(e, m+2)} A_n(ry)( r B_d(\gamma y,\eps) ) \\
		& \overset{\eqref{A-to-amenable-proof-disjoint}}{=} A_n(ry) \left( \bigsqcup_{\gamma \in B(e, m+2)}   r B_d(\gamma y,\eps) \right)  \overset{\eqref{A-to-amenable-proof-balls}}{=} 1 \, ,
	\end{align*}
	that is, $C_n^y$ is a probability measure. 
	\item For any $\gamma\in B(e, m+2)$ and any $y \in Y$, we have 
	\begin{align*}
		r B_d(\gamma y,\eps) \cap \operatorname{supp}(A_n(ry))  & \overset{\eqref{A-to-amenable-proof-balls}}{\subseteq} r B_d(\gamma y,\eps)  \cap \bigcup_{\eta \in B(e, m)} r B_d(\eta y,\eps/2)  \\
		& \overset{\eqref{A-to-amenable-proof-disjoint}}{\subseteq} r B_d(\gamma y,\eps/2) \, ,
	\end{align*}
	and thus for any set $D$ such that $r B_d(\gamma y,\eps/2) \subseteq D \subseteq r B_d(\gamma y,\eps)$, we have 
	\[
		C_n^y(\{\gamma\}) = A_n(ry) \left( r B_d(\gamma y,\eps) \right) = A_n(ry) \left( D \right)   \, .
	\] 
	If $d(y, y') < \eps/4$, the ball $D = r B_d(\gamma y, 3\eps/4)$ 
	satisfies both $r B_d(\gamma y,\eps/2) \subseteq D \subseteq r B_d(\gamma y,\eps)$ and $r B_d(\gamma y',\eps/2) \subseteq D \subseteq r B_d(\gamma y',\eps)$. Thus
	\[
		| C_n^{y'}(\{\gamma\}) - C_n^{y}(\{\gamma\}) | = \left | A_n(ry') \left( D \right) - A_n(ry) \left( D \right) \right| \leq \| A_n(ry') - A_n(ry) \|_1 \, ,
	\] 
	which shows that the map $y \mapsto C_n^y$ is continuous. 
	\item\label{A-to-amenable-proof-m} For any $\gamma\in B(e,m+2) \setminus B(e,m)$ and any $y \in Y$, we have
	\[
	r B_d(\gamma y,\eps) \cap \operatorname{supp}(A_n(ry))  \overset{\eqref{A-to-amenable-proof-balls}}{\subseteq} r B_d(\gamma y,\eps)  \cap \bigsqcup_{\eta \in B(e, m)} r B_d(\eta y,\eps/2)  \overset{\eqref{A-to-amenable-proof-disjoint}}{=} \varnothing \, ,
	\]
	which implies $C_n^y(\{\gamma\})=0$. 
	\item Since for $s\in S$ one has $B(e,m) \subseteq B(e,m+1)s \cap B(e,m+1)$ and also $B(e,m+1)s \cup B(e,m+1) \subseteq B(e,m+2)$, we check that $C_n$ is $\left(S, \frac{1}{n} \right)$-equivariant in the sense that for any generator $s \in S$ and any $y \in Y$, we have
	\begin{align*}
	\left\| C_n^y s^{-1} - C_n^{sy} \right\|_1 
	&\overset{\eqref{A-to-amenable-proof-m}}{=} \sum_{\gamma\in B(e,m+1)} \left|C_n^y(\gamma s) - C_n^{sy}(\gamma) \right| \\
	&= \sum_{\gamma\in B(e,m+1)} \left| A_n(ry)( r B_d(\gamma s y,\eps) ) - A_n(rsy)( r B_d(\gamma s y,\eps) )\right|\\
	&\overset{\eqref{A-to-amenable-proof-disjoint}}{\leq} \|A_n(ry) - A_n(rsy) \|_1 \overset{\eqref{A-to-amenable-proof-flat}}{\leq} 1/n\,. 
	\end{align*}
\end{enumerate}
Therefore the sequence of continuous maps $(y \mapsto C_n^y)$ satisfies the requirements in \cref{def:amenable-action}.
\end{proof}

Hence, combining \cref{Roe-amenable-actions} and \cref{A-to-amenable}, we obtain the following.

\begin{cor}\label{A-equivalence} Let $\Gamma\acts Y$ be a free action by Lipschitz homeomorphisms and assume that the infinite cone $\cO Y$ has property A (e.g.\ $Y$ is a manifold, simplicial complex, or profinite completion). Then, the action is amenable if and only if the warped cone $\cO_\Gamma Y$ has property A.
\end{cor}

\section{From embeddable warped cones to a-T-menable and \texorpdfstring{$PL^p$}{PL\^{}p} groups}\label{sec:embeddable}

In this section, we discuss generalisations of \cref{Roe3} from the introduction. The main result will be \cref{fibred-to-PL^p}, but let us start with the following analogue of \cref{A-to-amenable} for warped cones that do not necessarily have property A but admit a coarse embedding into the Hilbert space.

\begin{prop}\label{coarse-to-a-T-menable}
Let $\Gamma\acts Y$ be a free action of a finitely generated group on a compact metric space. If the warped cone $\cO_\Gamma Y$ admits a coarse embedding into a Hilbert space, then the action is a-T-menable.
\end{prop}

The necessity of the freeness assumption follows from \cref{non-a-T-menable-action}. Note also that our conclusion about the action cannot be strengthened to a conclusion about the group {\textemdash} the warped cone over an amenable action has property A, hence admits a coarse embedding into a Hilbert space, and there are plenty of such actions by non-a-T-menable groups. However, if an action is both a-T-menable and admits an invariant probability measure, then the group must be a-T-menable, so in this case we retrieve in particular the original version of \cref{Roe-embeddable} from \cite{Roe-cones}.

Let us recall the definition of a-T-menability of an action.

\begin{defi} An action $\Gamma\acts Y$ is a-T-menable if there exists a continuous function $h\colon\Gamma\times Y\to [0,\infty)$ which is:
\begin{itemize}
\item symmetric, that is, $h(\gamma, y) = h(\gamma^{-1},\gamma y)$;
\item normalised, that is, $h(e, y) = 0$;
\item proper; and
\item (conditionally) negative-definite, that is, for any $y\in Y$ and any vector $(\lambda_\gamma)\in \R^\Gamma$ with finite support and zero sum, the following holds
\begin{equation}\label{ineq:negative-definite}\sum_{\gamma, \gamma'\in\Gamma} \lambda_\gamma \lambda_{\gamma'} h(\gamma'\gamma^{-1},\, \gamma y) \leq 0\,.
\end{equation}
\end{itemize}
The reader may find these formulas more natural by regarding the pair $(\gamma, y)$ as the triple $(\gamma, y, \gamma y)$. 
\end{defi}

\begin{proof}[Proof of \cref{coarse-to-a-T-menable}]
Let $f$ be the coarse embedding and let $k\colon \cO_\Gamma Y \times \cO_\Gamma Y \to \R_+$ defined as $k(x,x')=\|f(x)-f(x')\|^2$ be the proper negative-type kernel satisfying \[\rho_-\circ d_\Gamma(x,x') \leq k(x,x') \leq \rho_+\circ d_\Gamma(x,x')\]
for some non-decreasing unbounded functions $\rho_-,\rho_+\colon \R_+\to\R_+$. By standard partition-of-unity arguments we can assume that $f$ and hence also $k$ are continuous.

If $y,y'\in Y$ are in the same orbit, and $\gamma$ is the shortest element of $\Gamma$ such that $\gamma y = y'$, then for $r\in\R_+$ large enough we have $d_\Gamma(ry,ry')=|\gamma|$ \cite{Sawicki-completions}*{Remark 3.1}. Hence, since the action is free, for any $y$ there is $r$ so large that $d_\Gamma(ry,r\gamma y)=|\gamma|$, and by the compactness of $Y$ there exists $r_\gamma$ valid for all $y$.

Let $(a_n)$ be an increasing sequence of integers such that $\rho_-(a_n)\geq 2^n \cdot n$ and let $r_n = \max \{r_\gamma \st \gamma \in B(e,a_{n+1}) \}$. We define
\[h(\gamma,y) = \sum_{n=1}^\infty \frac{1}{2^n} \cdot k(r_n y, r_n \gamma y)\,.\]
Since the kernel $k$ is negative-definite, the same is true for $h$, in the sense that it satisfies inequality \eqref{ineq:negative-definite}. We have $0\leq k(r_n y, r_n \gamma y)\leq \rho_+\circ d_\Gamma(r_n y, r_n \gamma y) \leq \rho_+(|\gamma|)$, so the series above indeed converges, and the map $y \mapsto h(\gamma, y)$ is bounded by $\rho_+(|\gamma|)$ and is continuous as a uniform limit of continuous functions. Finally, for any $n$, if $\gamma$ is long enough, that is, $a_n \leq |\gamma|$, and, say, $a_m\leq|\gamma|\leq a_{m+1}$ for some $m\geq n$, we obtain
\[h(y,\gamma)\geq \frac{k(r_m y, r_m \gamma y)}{2^m} \geq\frac{\rho_-\circ d_\Gamma(r_m y, r_m \gamma y)}{2^m} = \frac{\rho_-(|\gamma|)}{2^m} \geq \frac{\rho_-(a_m)}{2^m} \geq m \geq n\,,\]
so function $h$ is proper.
\end{proof}

We can give a similar proof in the presence of an invariant measure, but using affine actions rather than negative-type kernels. This way, we obtain a version of \cref{Roe-embeddable} that applies to any $L^p$-spaces.

\begin{prop}\label{coarse-to-PL^p}
Assume that the action $\Gamma\acts Y$ admits an invariant probability measure $\mu$ and is essentially free. Then, if the warped cone $\cO_\Gamma Y$ admits a coarse embedding into an $L^p$-space, then $\Gamma$ has property $PL^p$.
\end{prop}

\begin{proof}
Let $f$ be the coarse embedding of $\cO_\Gamma Y$ to some $L^p$-space $L$ and let $\rho_\pm$ be the control functions. As above, without loss of generality we can assume $f$ to be continuous, in particular measurable. Let $(a_n)$, similarly as in the previous proof, be such an increasing sequence of integers that $\rho_-(a_n)^p\geq 2^{n+1} \cdot n$. Recall that for any $y$ in a free orbit, the distance $d(ry,r\gamma y)$ equals $|\gamma|$ for $r$ large enough \cite{Sawicki-completions}. Hence, we can find a sequence $(r_n)$ increasing sufficiently fast that
\[\mu\left(\left\{y\in Y \st \forall\gamma\in \Gamma \big(a_n\leq |\gamma| < a_{n+1} \implies d(r_ny, r_n\gamma y)=|\gamma| \big) \right\}\right)>1/2\,.\]
Let us denote the above set by $Y_n$.

Let $m$ be the measure on the set $N=\{r_n \st n\in \N_+\}$ such that $m(\{r_n\}) = 2^{-n}$ and let $\nu$ denote the product measure $m \times \mu$ on $N \times Y \eqdef X$. We form a Bochner space $K=L^p(X,\nu; L)$, which comes with an isometric $\Gamma$-action: $(\gamma\cdot v)(r_n, y) = v(r_n, \gamma^{-1}y)$. Note that $K$ is an $L^p$-space. The cocycle for the action is given by $b(\gamma) = \gamma f - f$.

Let us check that the cocycle is correctly defined:
\begin{align*}
\|b(\gamma)\|_K^p &= \|\gamma f - f\|_K^p = \sum_n 2^{-n} \int_Y \|f(r_n\gamma^{-1} y) - f(r_ny)\|_L^p \, \dif\mu(y) \\
&\leq \sum_n 2^{-n} \rho_+(|\gamma^{-1}|)^p = \rho_+(|\gamma|)^p\,.
\end{align*}
To verify properness of the action, note that for any $m\in \N$ and $\gamma \in \Gamma$ such that $a_m \leq | \gamma | < a_{m+1}$, we have:
\begin{align*}
\|b(\gamma)\|_K^p
&= \sum_n 2^{-n} \int_Y \|f(r_n\gamma^{-1} y) - f(r_ny)\|_L^p \, \dif\mu(y) \\
&\geq 2^{-m} \int_Y \|f(r_m\gamma^{-1} y) - f(r_m y)\|_L^p \, \dif\mu(y) \\
&\geq 2^{-m} \int_{Y_m} \rho_- (|\gamma|)^p \, \dif\mu(y) \\
&\geq 2^{-m} \int_{Y_m} 2^{m+1} \cdot  m \, \dif\mu(y) \geq m\,.\qedhere
\end{align*}
\end{proof}

In fact, using the ultraproduct machinery developed by S. Arnt \cite{Arnt} for box spaces, we can strengthen \cref{coarse-to-PL^p} to the following.

\begin{theorem}\label{fibred-to-PL^p}
	Let $\Gamma\acts Y$ be an action of a finitely generated group on a compact metric space. 
Assume the action admits an invariant probability measure $\mu$ and is essentially free with respect to $\mu$. Then, if the warped cone $\cO_\Gamma Y$ admits a \emph{fibred} coarse embedding into an $L^p$-space, then $\Gamma$ has property $PL^p$.
\end{theorem}
\begin{proof}
Assuming the existence of a fibred coarse embedding of $\cO_\Gamma Y$ into some $L^p$-space $L$, we obtain from \cref{def:fibred} isometric copies $\{L_x\}_{x\in \cO_\Gamma Y}$ of $L$ as well as sequences of maps $f_n\colon \cO_\Gamma Y \to \bigsqcup_{x\in \cO_\Gamma Y} L_x$ and bounded sets $K_n\subseteq \cO_\Gamma Y$. We let $(r_n)$ be an unbounded increasing sequence of positive reals such that $K_n \cap \big([r_n,\infty) \times Y\big) = \emptyset$. We also have the trivialising maps $\tau_n^x$, $t_n^x$, and $t_n^{x,x'}$ as in \cref{def:fibred}. By a partition-of-unity argument we can assume measurability of the map $(x,x')\mapsto t_n^x \circ f_n(x')$ as before.

Fix $n\in \N_+$. We will identify the set $Y$ with $\{r_n\}\times Y \subseteq \cO_\Gamma Y$, which allows us to treat $f_n$ as a function on $Y$.

We define local cocycles $c_n\colon B(e,n)\to L^p(Y,\mu; L)$ similarly as in the previous proof:
\[c_n(\gamma)(y) =  t^y_n \circ f_n(\gamma^{-1} y) - t^y_n \circ f_n(y)\]
(we use here the fact that $d_\Gamma(x,\gamma^{-1} x)\leq |\gamma^{-1}|\leq n$, so $f_n(\gamma^{-1}y)$ lies in the domain of $t^y_n$). We also need to correct our ``representation'', which also will be defined only for $\gamma\in B(e,n)$:
\[(\gamma^{-1} v)(y) = T^{y,\gamma y}_n v(\gamma y)\,,\]
where $T^{y,\gamma y}_n$ is the linear part of the affine isometry $t_n^{y,\gamma y} \colon L\to L$ from \cref{def:fibred}.

Let us check that we have defined a local homomorphism:
\begin{align*}
(\eta^{-1}(\gamma^{-1} v))(y)
&= T^{y,\eta y}_n (\gamma^{-1} v) (\eta y) = T^{y,\eta y}_n T^{\eta y, \gamma\eta y}_n v(\gamma\eta y) \\
&= T^{y, \gamma\eta y}_n v(\gamma\eta y) = ((\eta^{-1}\gamma^{-1}) v) (y)\,,
\end{align*}
where we used the equality $T^{y,\eta y}_n T^{\eta y, \gamma\eta y}_n = T^{y, \gamma\eta y}_n$, which follows from the equality of the respective affine isometries:
\begin{align*}
\id_{I} \times \left( t^{y,\eta y}_n \circ t^{\eta y, \gamma\eta y}_n \right) 
&= \tau_n^y\circ (\tau_n^{\eta y})^{-1} \circ \tau_n^{\eta y}\circ (\tau_n^{\gamma \eta y})^{-1} \\
&= \tau_n^y \circ (\tau_n^{\gamma \eta y})^{-1} \\
&= \id_I \times t^{y,\gamma\eta y}_n ,
\end{align*}
which in turn uses the non-emptiness of $I = B(y,n)\cap B(\eta y, n) \cap B(\gamma\eta y, n)$.
Now, we can check that $c_n$ is indeed a local cocycle with respect to the above local representation. First, we observe that:
\begin{align*}
(\gamma^{-1} c_n(\eta^{-1}))(y)
&= T^{y,\gamma y}_n c_n(\eta^{-1})(\gamma y)\\
&= T^{y,\gamma y}_n \left(t^{\gamma y}_n \circ f_n(\eta \gamma y) - t^{\gamma y}_n \circ f_n(\gamma y) \right) \\
&= t^{y,\gamma y}_n \left(t^{\gamma y}_n \circ f_n(\eta \gamma y) - t^{\gamma y}_n \circ f_n(\gamma y)\right) \\
&= t^{y}_n \circ f_n(\eta \gamma y) - t^{y}_n \circ f_n(\gamma y)\,,
\end{align*}
and hence we get the cocycle condition:
\begin{align*}
(\gamma^{-1}c_n(\eta^{-1}))(y) + c_n&(\gamma^{-1})(y) \\
&= t^{y}_n \circ f_n(\eta \gamma y) - t^{y}_n \circ f_n(\gamma y) 
+ t^y_n \circ f_n(\gamma y) - t^y_n \circ f_n(y)\\
&=  t^{y}_n \circ f_n(\eta \gamma y) - t^y_n \circ f_n(y)
= c_n(\gamma^{-1}\eta^{-1})(y)\,.
\end{align*}

Recall that $d_\Gamma(r_n y, r_n \gamma y)$ is always bounded by $|\gamma|$ irrespective of $n$, so we have 
\[\|c_n(\gamma)(y)\|_L
=\| t^y_n \circ f_n(\gamma^{-1} y) - t^y_n \circ f_n(y)\|_L
\leq \rho_+(|\gamma|)\,,\]
and also $d_\Gamma(r_n y, r_n \gamma y)$ equals $|\gamma|$ for almost every $y$ and $n$ large enough, and then we get $\|c_n(\gamma)(y)\|_L \geq \rho_-(|\gamma|)$, in particular $\lim_n \|c_n(\gamma)\| \geq \rho_-(|\gamma|)$.

Consider now the space $\ell^\infty(\N; L^p(Y,\mu;L))$, a non-principal ultrafilter $\U$, and the quotient space $\LL = \ell^\infty(\N; L^p(Y,\mu;L))/c_0^\U(\N; L^p(Y,\mu;L))$, called an ultraproduct, where
\[c_0^\U(\N; L^p(Y,\mu;L)) = \{(l_n)\in \ell^\infty(\N; L^p(Y,\mu;L)) \st \lim_\U \|l_n\|_{L^p} = 0\}\,.\]
The norm on $\LL$ is given by the ultralimit $\|(l_n) \|_\LL = \lim_\U \|l_n\|_{L^p}$.

We can extend our local isometric representations $B(e,n)\to \mathrm{Isom}(L^p(Y,\mu;L))$ to maps $\Gamma \to \mathrm{Isom}(L^p(Y,\mu;L))$ by the identity outside $B(e,n)$ and consider the product map $\Gamma\to \mathrm{Isom}(\ell^\infty(\N,L^p(Y,\mu;L)))$ and similarly extend the cocycles $c_n\colon B(e,n)\allowbreak \to L^p(Y,\mu;L)$ by zero outside $B(e,n)$ and consider the product map into the space $\ell^\infty(\N,L^p(Y,\mu;L))$ (we use the uniform bound $\|c_n\|\leq \rho_+(|\gamma|)$). It is easy to check that, after passing to the quotient space $\LL$, the first map becomes an isometric representation on $\LL$, and the second is a cocycle with respect to it (see \cite{Arnt}*{Lemma~3.3}). We also have
\[\rho_-(|\gamma|) \leq \|(c_n)\|_\LL \leq \rho_+(|\gamma|)\,,\]
so the obtained affine isometric action is indeed proper. The claim follows from the fact that an ultraproduct of $L^p$-spaces is again an $L^p$-space.
\end{proof}

\section{The importance of freeness}\label{sec:freeness}

In this section we provide examples (\cref{ball-action} and~\labelcref{non-amenable-action}) showing why one needs the freeness assumption in \cref{Haagerup-to-fibred} (and \cref{PL^p-to-fibred}) and in \cref{A-to-amenable} (and \cref{coarse-to-a-T-menable}). It turns out that there exist counterexamples with only \emph{one} fixed point, where the action is free on its complement (and all the other assumptions of the respective theorems are satisfied). Interestingly, these two theorems provide implications in ``opposite directions'' (in the former we deduce a property of the warped cone from a property of the group and in the latter we deduce a property of the action from a property of the warped cone). It is also worth observing that for the converse implications (respectively \cref{fibred-to-PL^p} and \cref{Roe-amenable-actions}) freeness is not required.

Let $CY = [0,1] \times Y / \{0\} \times Y$ be the compact cone over $Y$ equipped with the metric $d((\theta, y), (\theta', y')) = |\theta-\theta'| \cdot \operatorname{diam}(Y) + \min(\theta,\theta') \cdot d(y,y')$ (where we denote the metric on $Y$ and $CY$ with the same letter $d$). 

\begin{prop}\label{freeness-embeddability}
The following conditions are equivalent for $\Gamma\acts Y$:
\begin{enumerate}
\item\label{freeness-e-fibred-item} $\cO_\Gamma CY$ fibred coarsely embeds into a Hilbert space (respectively: Banach space $E$);
\item\label{freeness-e-CY-item} $\cO_\Gamma CY$ coarsely embeds into a Hilbert space (respectively: Banach space $E$);
\item\label{freeness-e-Y-item} $\cO_\Gamma Y$ coarsely embeds into a Hilbert space (respectively: Banach space $E$);
\end{enumerate}
where the Banach space version requires that $Y$ is a geodesic space (up to bi-Lipschitz equivalence) and $E$ contains an isomorphic copy of $E \oplus E$ .
\end{prop}
\begin{proof}
Let $d_r$ denote the warped metric coming from the metric $rd$ on $Y$ or $CY$. It is proved in \cite{Sawicki-completions}*{Lemma 4.1} that the warped cone $\cO_\Gamma Y$ embeds coarsely into a Hilbert space if and only if the family of metric spaces $((Y,d_r))_{r>0}$ embeds coarsely into a Hilbert space in a uniform way. If $(Y,d)$ is a geodesic space, then the same holds for Banach spaces as above, see \cite{Sawicki-completions}*{Corollary 4.2} (this result is stated for $L^p$-spaces, but essentially the same proof works under our assumption).

When the considered $\Gamma$-space is itself a cone $CY$, we can apply the argument twice: the embeddability of $\cO_\Gamma CY$ is equivalent to the embeddability of the family $\{(CY,d_{s})\st s>0\}$, which in turn is equivalent to the embeddability of
\begin{equation}\label{eq:sections-double-cone}
\{(Y,d_{\theta\cdot s}) \st s>0,\, \theta\in(0,1]\} = \{(Y,d_{r}) \st r>0 \}\,.
\end{equation}
The last equality immediately proves the equivalence of conditions \eqref{freeness-e-CY-item} and \eqref{freeness-e-Y-item}.

Now, assuming \eqref{freeness-e-fibred-item} for a Hilbert or Banach space $E$, there exist non-decreasing and unbounded functions $\rho_\pm\colon [0,\infty) \to [0,\infty)$ and, for any $n$, a family of maps $F_n^z\colon B(z,n)\to E$ satisfying
\[\rho_-\circ d_\Gamma (x,x')\leq \|F_n^z(x) - F_n^z(x') \|\leq \rho_+\circ d_\Gamma (x,x')\]
for all $z$ in a complement of a bounded set $K_n \subseteq \cO_\Gamma CY$. There is some $r_n<\infty$ such that $[r_n, \infty) \times CY \subseteq \cO_\Gamma CY \setminus K_n$.

We will prove coarse embeddability of the family $\{(Y,d_{r}) \st r>0 \}$ with uniform estimates by $\rho_\pm$, which is equivalent to \eqref{freeness-e-Y-item}. Given $r>0$, let $R = \max\left(r_{\lceil\diam Y \cdot r\rceil}, r\right)$ and $\theta=\frac{r}{R}$. Then, $(Y,d_r)$ embeds isometrically into $(CY,d_R)$ as $\{\theta\} \times Y$ (consult \cite{Sawicki-completions}*{Lemma~1.6}), and $F_R^{z_0}$ yields a coarse embedding for any $z_0\in \{\theta\} \times Y \subseteq (CY,d_R)$.
\end{proof}

\begin{ex}\label{ball-action} Let $\Gamma$ be a free subgroup of $\mathrm{SU}(2)\simeq \SS^3 \subseteq \C^2$ such that the action $\Gamma\acts \SS^3$ has a spectral gap as established by \cite{BG} and let $B$ denote the unit ball in $\C^2$. Then the warped cone $\cO_\Gamma B$ does not fibred coarsely embed into any $L^p$-space for $1\leq p < \infty$ even though $\Gamma\acts B$ is a linearisable action of an a-T-menable group which is free after removing one fixed point.
\end{ex}
\begin{proof}
We know that $\cO_\Gamma \SS^3$ does not embed coarsely into any $L^p$-space by the spectral gap property \cite{Nowak--Sawicki}. By \cref{freeness-embeddability} this implies that $\cO_\Gamma C\SS^3\simeq \cO_\Gamma B$ does not fibred coarsely embed into these spaces.

Linearisability of the action is obvious for the Hilbert space case and follows from \cref{manifold} applied to the action on $\SS^3$ for $p\neq 2$.
\end{proof}

Clearly, the same can be said for any subgroup action on an ambient compact Lie group if the action has a spectral gap, and the subgroup is a-T-menable. Such examples abound by~\cite{BdS}.

Similarly, one can obtain the following.

\begin{prop}\label{freeness-amenability}
The following conditions are equivalent for $\Gamma\acts Y$:
\begin{enumerate}
\item\label{freeness-a-CY-item} $\cO_\Gamma CY$ has property A;
\item\label{freeness-a-Y-item} $\cO_\Gamma Y$ has property A;
\item\label{freeness-a-Y+item} $\cO_\Gamma Y^+$ has property A, where $Y^+ = Y\sqcup \{*\}$ and the action on $\{*\}$ is trivial.
\end{enumerate}
\end{prop}
\begin{proof} The equivalence of \eqref{freeness-a-CY-item} and \eqref{freeness-a-Y-item} follows from the fact that property A of a warped cone $\cO_\Gamma Y$ is equivalent to uniform property A of its sections $\{(Y,d_{r}) \st r>0 \}$, {Proposition~5.2} in \cite{Sawicki-completions}, and \cref{eq:sections-double-cone} on page~\labelcpageref{eq:sections-double-cone}.

Similarly, the equivalence of \eqref{freeness-a-Y-item} and \eqref{freeness-a-Y+item} follows from an easy exercise that the family $\{(Y,d_{r}) \st r>0 \}$ has property A in a uniform way if and only if the family $\{(Y^+,d_{r}) \st r>0 \}$ does.
\end{proof}

\begin{ex}\label{non-amenable-action} Let $\Gamma\acts Y$ be a free, Lipschitz, and amenable action of a non-amenable finitely generated group on a compact $Y\subseteq \R^n$. Then, the warped cone $\cO_\Gamma CY$ has property A even though the action $\Gamma \acts CY$ is not amenable, and similarly for the warped cone $\cO_\Gamma Y^+$ and the action $\Gamma \acts Y^+$.
\end{ex}
\begin{proof}
Property A of $\cO_\Gamma Y$ follows from \cref{Roe-amenable-actions}. By the above \cref{freeness-amenability}, it implies also property A of $\cO_\Gamma CY$ and $\cO_\Gamma Y^+$. However, $CY$ and $Y^+$ contain a fixed point for the action of $\Gamma$, and the action on a point is amenable if and only if the group is amenable.
\end{proof}

If we require $\Gamma$ to be non-a-T-menable, we can obtain a stronger version of \cref{non-amenable-action}.

\begin{ex}\label{non-a-T-menable-action}
There exist warped cones with property A over actions which are not even a-T-menable.
\end{ex}

\section{Open problems}\label{sec:problems}

In the proof of \cref{coarse-to-a-T-menable} instead of the average $h=\sum\frac{h_n}{2^n}$ of kernels $h_n(\gamma, y) \defeq k(r_ny, r_n\gamma y)$ one may consider their ultralimit. The limit kernel is still proper and negative definite (in the sense of inequality \eqref{ineq:negative-definite} on page~\labelcpageref{ineq:negative-definite}), even if we assume that the warped cone admits only an asymptotic coarse embedding instead of a coarse embedding. However, we cannot see any reason, why it would be continuous, which leaves the following question open.

\begin{question} 
	Is a free action $\Gamma\acts Y$ always a-T-menable if $\cO_\Gamma Y$ admits an asymptotic coarse embedding into a Hilbert space?
\end{question}

It is also tempting to ask about the converse, since it is true (\cref{Haagerup-to-fibred}) for actions admitting an invariant measure (a-T-menability of the action is then equivalent to a-T-menability of the group), and its analogue, \cref{Roe-amenable-actions}, is valid for amenable actions and warped cones with property A.

\begin{question} Does a-T-menability of a free action $\Gamma\acts Y$ imply fibred or asymptotic coarse embeddability of $\cO_\Gamma Y$ into a Hilbert space (assuming that $\cO Y$ itself is embeddable)?
\end{question}

If we assume that the above action is isometric, then it typically also admits some invariant Hausdorff measure, and we are back in the case of an a-T-menable group. Hence, if one wants to give a positive answer to the above, one should first study the following question.

\begin{question} Does a warped cone over a free action by \emph{Lipschitz} homeomorphisms of an a-T-menable group admit an asymptotic coarse embedding or a fibred coarse embedding?
\end{question}

\begin{center}
\medskip This is a preprint version of the article \href{https://doi.org/10.1142/S179352532050034X}{DOI 10.1142/S179352532050034X} \\published by the Journal of Topology and Analysis.
\end{center}

\end{document}